\documentclass[english,4paper,10pt]{article}
\usepackage[utf8]{inputenc}
\usepackage[T1]{fontenc}
\usepackage[english]{babel}
\usepackage[autostyle=true]{csquotes}
\usepackage{mathdots}
\usepackage{latexsym, amssymb, amsmath}
\usepackage{amsthm}
\makeatletter
\renewcommand*\env@matrix[1][*\c@MaxMatrixCols c]{%
	\hskip -\arraycolsep
	\let\@ifnextchar\new@ifnextchar
	\array{#1}}
\usepackage{url}
\usepackage{float}
\usepackage{caption}
\usepackage{subcaption}
\usepackage{color}   
\usepackage{hyperref}
\usepackage[symbol]{footmisc}
\hypersetup{
	colorlinks=true, 
	linkcolor=black,
	filecolor=magenta,      
	urlcolor=cyan,
	linktoc=all,     
}
\usepackage[ruled,vlined]{algorithm2e}
\usepackage{graphicx}
\graphicspath{ {./images/} }
\usepackage{amscd}

\newcommand{\R}{\mathbb{R}}

\newcommand{\F}{\mathbb{F}}

\newtheorem{thm}{Theorem}[section]

\newtheorem{definition}[thm]{Definition}
\newtheorem{rem}[thm]{Remark}
\newtheorem{prop}[thm]{Proposition}
\newtheorem{ex}[thm]{Example}

\makeatletter
\def\blfootnote{\xdef\@thefnmark{}\@footnotetext}
\makeatother

\date{}

\usepackage[%
backend=bibtex,%
style=numeric-comp,%
sorting=nyt,%
hyperref,%
]{biblatex}

\DeclareUnicodeCharacter{2212}{-}
\begin{document}
	\sloppy
	
	\title{Derivations of two-step nilpotent algebras \blfootnote{\textit{\textup{2020} Mathematics Subject Classification}: 15B30, 16W25, 17A32, 17A36, 17B40.}} \maketitle
	\noindent
	{{Gianmarco La Rosa}\footnote[2]{Supported by University of Palermo and by the “National Group for Algebraic and Geometric Structures, and their Applications” (GNSAGA – INdAM).}, {Manuel Mancini}\footnotemark[2]\\
%		$^,$\footnote[1]{Corresponding Author.} \\
		\footnotesize{Dipartimento di Matematica e Informatica}\\
		\footnotesize{Universit\`a degli Studi di Palermo, Via Archirafi 34, 90123 Palermo, Italy}\\
		\footnotesize{gianmarco.larosa@unipa.it}, ORCID: 0000-0003-1047-5993 \\
		\footnotesize{manuel.mancini@unipa.it}, ORCID: 0000-0003-2142-6193}

\begin{abstract}
In this paper we study the Lie algebras of derivations of two-step nilpotent algebras. We obtain a class of Lie algebras with trivial center and abelian ideal of inner derivations. Among these, the relations between the complex and the real case of the indecomposable \emph{Heisenberg Leibniz algebras} are thoroughly described. Finally we show that every \emph{almost inner derivation} of a complex nilpotent Leibniz algebra with one-dimensional commutator ideal, with three exceptions, is an inner derivation.
\end{abstract}

\section*{Introduction}

Leibniz algebras were first introduced by J.-L. Loday in \cite{loday1993version} as a non-antisymmetric version of Lie algebras, and many results of Lie algebras were also established in the frame of Leibniz algebras. Earlier, such algebraic structures were considered by A. Blokh, who called them D-algebras \cite{blokhLie}, for their strict connection with the derivations. Leibniz algebras play a significant role in different areas of mathematics and physics. \\
In \cite{BarDiFal18}  and \cite{falcone2016action} derivations of nilpotent Lie algebras with small dimensional commutator ideal were studied; in \cite{Katgorodov1} and \cite{LM2022} two-step nilpotent algebras were classified. After a short preliminary section, in this paper we aim to study the Lie algebras of derivations of two-step nilpotent algebras. Independently of its intrinsic interest, derivations find concrete applications in representation theory (cf.\ \cite{nonabext}), (sub-)Riemannian geometry and control theory (see \cite{BiggsNagy}, and the bibliography therein), just to give two instances. 

The first section is devoted to some background material on Leibniz algebras and derivations which will be useful for the rest of the paper. We address the reader to \cite{ayupov2019leibniz} and \cite{erdmann2006introduction} for more details. We recall that there are only three classes of indecomposable nilpotent Leibniz algebras with one-dimensional commutator ideal, namely the \emph{Heisenberg Leibniz algebras} $\mathfrak{l}_{2n+1}^A$, parametrized by their dimension $2n+1$ and a $n \times n$ matrix $A$ in canonical form, the \emph{Kronecker Leibniz algebras} $\mathfrak{k}_n$ and the \emph{Dieudonné Leibniz algebras} $\mathfrak{d}_n$, both parametrized by their dimension only. Recently noncommutative Heisenberg algebras were also studied in \cite{Kaygorodov2}.

In Sections 2, 3, 4 we find the derivation algebras of such nilpotent Leibniz algebras. The properties of these Lie algebras are studied and, for the Heisenberg algerbras, we analyze the relations between the complex and the real case. We always obtain Lie algebras of derivations with trivial center and abelian ideal of inner derivations. We use the \emph{LieAlgebras} package of the software \emph{Maple} for finding the radical, the Levi decomposition and the nilradical of such algebras.

In the last section we show that, for a complex nilpotent Leibniz algebra $L$ with one-dimensional commutator ideal, such that $L$ is not isomorphic to $\mathfrak{l}_{2n+1}^{J_{\pm 1}}$ or $\mathfrak{d}_n$, every \emph{almost inner derivation} is an inner derivation. This generalizes the result found by D. Burde, K. Dekimpe and V. Verbeke in \cite{Burde} for two-step nilpotent Lie algebras of \emph{genus} $1$. We also find the Lie subalgebras $\operatorname{AIDer}(\mathfrak{l}_{2n+1}^{J_{a}})$, when $a=\pm 1$, and $\operatorname{AIDer}(\mathfrak{d}_{n})$, showing that in this case the inclusions
$$
\operatorname{AIDer}(\mathfrak{l}_{2n+1}^{J_{\pm 1}}) \supsetneq \operatorname{Inn}(\mathfrak{l}_{2n+1}^{J_{\pm 1}})
\text{ and }
\operatorname{AIDer}(\mathfrak{d}_{n}) \supsetneq \operatorname{Inn}(\mathfrak{d}_{n})
$$
are strict.

\section{Preliminaries}

We assume that $\F$ is a field with $\operatorname{char}(\F)\neq2$. For the general theory we refer to \cite{ayupov2019leibniz}.

\begin{definition}
	A \emph{left Leibniz algebra} over $\mathbb{F}$ is a vector space $L$ over $\mathbb{F}$ endowed with a bilinear map (called $commutator$ or $bracket$) $\left[-,-\right]\colon L\times L \rightarrow L$ which satisfies the \emph{left Leibniz identity}
	$$
	\left[x,\left[y,z\right]\right]=\left[\left[x,y\right],z\right]+\left[y,\left[x,z\right]\right], \; \; \forall x,y,z\in L.
	$$
	
\end{definition}
	In the same way we can define a right Leibniz algebra, using the right Leibniz identity
	$$
	\left[\left[x,y\right],z\right]=\left[[x,z],y\right]+\left[x,\left[y,z\right]\right], \; \;\forall x,y,z\in L.
	$$
	A Leibniz algebra that is both left and right is called \emph{symmetric Leibniz algebra}. From now on we assume that $\dim_\F L<\infty$.\\
	
	 Clearly every Lie algebra is a Leibniz algebra and every Leibniz algebra with skew-symmetric commutator is a Lie algebra. Thus it is defined an adjunction (see \cite{mac2013categories}) between the category $\textbf{LieAlg}_{\F}$ of the Lie algebras over $\F$ and the category $\textbf{LeibAlg}_{\F}$ of the Leibniz algebras over $\F$. The left adjoint of the full inclusion $i \colon \textbf{LieAlg}_{\F}\rightarrow \textbf{LeibAlg}_{\F}$ is the functor $\pi \colon \textbf{LeibAlg}_{\F} \rightarrow \textbf{LieAlg}_{\F}$ that associates with every Leibniz algebra $L$ the quotient $L/\operatorname{Leib}(L)$, where $\operatorname{Leib}(L)=\langle \left[x,x\right] |\,x\in L \rangle$. We observe that $\operatorname{Leib}(L)$ is the smallest bilateral ideal of $L$ such that $L/\operatorname{Leib}(L)$ is a Lie algebra. Moreover $\operatorname{Leib}(L)$ is an abelian algebra. \\
	 
	  The definition of \emph{derivation} is the same as in the case of Lie algebras.
	  \begin{definition}
	  	Let $L$ be a Leibniz algebra over $\F$. A \emph{derivation} of $L$ is a linear map $d \colon L\rightarrow L$ such that 
	  	$$
	  	d(\left[x,y\right])=\left[d(x),y\right]+\left[x,d(y)\right], \;\; \forall x,y\in L.
	  	$$
	  \end{definition}
	  
	  The left multiplications are particular derivations called \emph{inner derivations} and an equivalent way to define a left Leibniz algebra is to say that the (left) adjoint map $\operatorname{ad}_x=\left[x,-\right]$ is a derivation, for every $x\in L$. Meanwhile, for a left Leibniz algebra, the right adjoint maps are not derivations in general.\\ With the usual bracket \hbox{$\left[d_1,d_2\right]=d_1\circ d_2 - d_2\circ d_1$}, the set $\operatorname{Der}(L)$ is a Lie algebra and the set  $\operatorname{Inn}(L)$ of all inner derivations of $L$ is an ideal of $\operatorname{Der}(L)$. Furthermore, $\operatorname{Aut}(L)$ is a Lie group and the associated Lie algebra is $\operatorname{Der}(L)$.
	  \begin{definition}
	  	A derivation $d$ of a Leibniz algebra $L$ is called \emph{almost inner derivation} if $d(x) \in [L,x]$, for every $x \in L$.
	  \end{definition}
  
   In \cite{AIDer} Z.\ K.\ Adashev and T.\ K.\ Kurbanbaev proved that the space of all almost inner derivations, denoted by $\operatorname{AIDer}(L)$, is a Lie subalgebra of $\operatorname{Der}(L)$ and 
   $$
   \operatorname{Inn}(L) \subseteq \operatorname{AIDer}(L).
   $$
	We define the left and the right center of a Leibniz algebra
	$$
	\operatorname{Z}_l(L)=\left\{x\in L \,|\,\left[x,L\right]=0\right\},\,\,\, \operatorname{Z}_r(L)=\left\{x\in L \,|\,\left[L,x\right]=0\right\},
	$$
	and we observe that they coincide when $L$ is a Lie algebra. The \emph{center} of $L$ is $\operatorname{Z}(L)=\operatorname{Z}_l(L)\cap \operatorname{Z}_r(L)$. In the case of symmetric Leibniz algebras, the left center and the right center are bilater ideals, but in general $\operatorname{Z}_l(L)$ is an ideal of the left Leibniz algebra $L$, meanwhile the right center is not even a subalgebra.
	
	\begin{definition}
		Let $L$ be a left Leibniz algebra over $\F$ and let
		$$
		L^{(0)}=L,\,\,L^{(k+1)}=[L,L^{(k)}],\; \; \forall k\geq0,
		$$ be the \emph{lower central series of $L$}. $L$ is \emph{$n-$step nilpotent} if $L^{(n-1)}\neq0$ and $L^{(n)}=0$.
	\end{definition}

We observe that two-step nilpotent Leibniz algebras lie in different varieties of non-associative algebras, such as associative, alternative and Zinbiel algebras. For this reason we will refer at them as \emph{two-step nilpotent algebras}.

One can directly prove the following.
	
	\begin{prop}
		If $L$ is a left two-step nilpotent algebra, then $L^{(1)}=[L,L]\subseteq \operatorname{Z}(L)$ and $L$ is symmetric.
	\end{prop}
	
	\begin{prop}
		If $L$ is a left nilpotent Leibniz algebra with $\dim_\F L^{(1)}=1$, then  $L^{(1)}\subseteq \operatorname{Z}(L)$ and $L$ is symmetric.
	\end{prop}

Nilpotent Leibniz algebras with one-dimensional commutator ideal were classified in \cite{LM2022}. It was proved that, up to isomorphism, there are only the following three classes of indecomposable nilpotent Leibniz algebras with one-dimensional commutator ideal. 
  
\begin{definition}\cite{LM2022}
Let $f(x)\in\F\left[x\right]$ be a monic irreducible polynomial. Let $k\in\mathbb{N}$ and let $A=(a_{ij})_{i,j}$ be the companion matrix of $f(x)^k$. The \emph{Heisenberg} Leibniz algebra $\mathfrak{l}_{2n+1}^A$ is the $(2n+1)$-dimensional Leibniz algebra with basis $\left\{e_1,\ldots,e_n,f_1,\ldots,f_n,z\right\}$ and the Leibniz brackets are given by
$$
[e_i,f_j]=(\delta_{ij}+a_{ij})z, \; [f_j,e_i]=(-\delta_{ij}+a_{ij})z, \; \; \forall i,j=1,\ldots,n.
$$
\end{definition}

Notice that, if $A$ is the zero matrix, then we obtain the $(2n+1)-$dimensional Heisenberg Lie algebra $\mathfrak{h}_{2n+1}$.

\begin{definition}\cite{LM2022}
	Let $n \in \mathbb{N}$. The \emph{Kronecker} Leibniz algebra $\mathfrak{k}_{n}$ is the $(2n+1)$-dimensional Leibniz algebra with basis $\left\{e_1,\ldots,e_n,f_1,\ldots,f_n,z\right\}$ and the Leibniz brackets are given by
	\begin{align*}
		&[e_i,f_i] = [f_i,e_i] = z, \; \; \forall i=1,\ldots,n \\
		& [e_i,f_{i-1}] = z, [f_{i-1},e_i] = −z, \; \; \forall i= 2,\ldots,n.
	\end{align*}
\end{definition}

\begin{definition}\cite{LM2022}
	Let $n \in \mathbb{N}$. The \emph{Dieudonné} Leibniz algebra $\mathfrak{d}_n$ is the $(2n+2)$-dimensional Leibniz algebra with basis $\left\{e_1,\ldots,e_{2n+1},z\right\}$ and and the Leibniz brackets are given by
	\begin{align*}
		&\left[e_1,e_{n+2}\right]=z,\\
		&\left[e_i,e_{n+i}\right]=\left[e_i,e_{n+i+1}\right]=z, \; \; \forall i=2,\ldots,n,\\
		&\left[e_{n+1},e_{2n+1}\right]=z,\\
		&\left[e_i,e_{i-n}\right]=z, \; \; \left[e_i,e_{i-n-1}\right]=-z, \; \; \forall i=n+2,\ldots,2n+1.
	\end{align*}
	
\end{definition}

In the next sections we study the Lie algebra of derivations of the three classes of indecomposable nilpotent Leibniz algebras with one-dimensional commutator ideal. We observe that, given a derivation $d$ of a Leibniz algebra $L$, we have
$$
d([L,L])\subseteq [L,L],
$$
thus, if $[L,L]=\F z$, it follows that $d(z)=\gamma z$, for some $\gamma \in \F$.
	
	\section{Derivations of the Heisenberg Leibniz algebras $\mathfrak{l}_{2n+1}^A$}
	
	Now we want to study in details the Lie algebras of derivations of the Heisenberg Leibniz algebras in the case the field $\F$ is $\mathbb{C}$ or $\R$.
	
	\subsection{The complex case}
	
	Let $n\in\mathbb{N}$ and let $f(x)=x-a\in\mathbb{C}[x]$. Then the companion matrix $A$ of $f(x)^n$ is similar to the $n\times n$ Jordan block $J_a$ of eigenvalue $a$. Thus $\mathfrak{l}_{2n+1}^A\cong\mathfrak{l}_{2n+1}^{J_a}$ and the Leibniz brackets are given by
	\begin{align*}
		&\left[e_i,f_i\right]=(1+a)z, \; \; \left[f_i,e_i\right]=(-1+a)z, \; \; \forall i=1,\ldots,n,\\
		&\left[e_i,f_{i-1}\right]=\left[f_{i-1},e_i\right]=z, \; \; \forall i=2,\ldots,n,
	\end{align*}
	where $\left\{e_1,\ldots,e_n,f_1, \ldots,f_n,z\right\}$ is a basis of $\mathfrak{l}_{2n+1}^{J_a}$. Moreover $\mathfrak{l}_{2n+1}^{J_a} \cong \mathfrak{l}_{2n+1}^{J_{-a}}$.
	
	Now let $d\colon \mathfrak{l}_{2n+1}^{J_a} \rightarrow \mathfrak{l}_{2n+1}^{J_a}$ be a linear endomorphism such that
	$$
	d(e_i)=\sum_{j=1}^{n}a_{ji}e_j + \sum_{k=1}^{n}b_{ki}f_k + a_i z, \; \; \forall i=1,\ldots,n,
	$$
	$$
	d(f_i)=\sum_{j=1}^{n}c_{ji}e_j + \sum_{k=1}^{n}d_{ki}f_k + b_i z,  \; \; \forall i=1,\ldots,n
	$$
	and
	$$
	d(z)=\gamma z.
	$$
	Then for $a \neq 0$ the linear endomorphism $d$ is a derivation of the Heisenberg algebra $\mathfrak{l}_{2n+1}^{J_a}$ if and only if it has the following form
	$$
	\begin{pmatrix}
		A & 0 & 0 \\
		0 & D & 0 \\
		\mu & \nu & \gamma \\
	\end{pmatrix}
    $$
    where
    $$
    A=\left(
    \begin{array}{*5{c}}
    	\alpha_1 & \alpha_2 & \alpha_3& \ldots & \alpha_n \\
    	0 & \alpha_1 & \alpha_2 & \ldots & \alpha_{n-1}  \\
    	0 &  0 &\alpha_1 & \ldots & \alpha_{n-2}  \\
    	\vdots &  \vdots &  \vdots & \ddots & \vdots \\
    	0 &  0 &  0 &  \ldots & \alpha_1 \\
    \end{array}\right), \; \;
D= \left(
\begin{array}{*5{c}}
	\beta_1 & 0 & 0 & \ldots & 0 \\
	-\alpha_2 & \beta_1 & 0 & \ldots & 0  \\
	-\alpha_3 &  -\alpha_2 &\beta_1 & \ldots & 0  \\
	\vdots &  \vdots &  \vdots & \ddots & \vdots \\
	-\alpha_n &  -\alpha_{n-1} &  -\alpha_{n-2} &  \ldots & \beta_1 \\
\end{array}\right),
    $$
    $$
    \mu=(\mu_1,\mu_2,\mu_3,\ldots,\mu_n), \; \; \nu=(\nu_1,\nu_2,\nu_3,\ldots,\nu_n), \; \; \gamma=\alpha_1+\beta_1.
    $$ 
    Thus, $\operatorname{Der}(\mathfrak{l}_{2n+1}^{J_a})$ is the Lie subalgebra of $\operatorname{gl}(2n+1,\mathbb{C})$ of dimension $3n+1$ with basis
    $$
    \lbrace x,y,E_1,\ldots,E_{n-1},A_1,\ldots,A_n,B_1,\ldots,B_n \rbrace,
    $$
    where
    \begin{align*}
    	&x=\sum_{k=1}^{n} e_{k,k}+e_{2n+1,2n+1}, \; \; y=\sum_{k=1}^{n} e_{n+k,n+k}+e_{2n+1,2n+1},\\
    	&E_i=\sum_{k=1}^{n-i}(e_{k,k+i}-e_{n+i+k,n+k}), \; \; \forall i=1,\ldots,n-1,\\
    	&A_i=e_{2n+1,i}, \; \; B_i=e_{2n+1,n+i}, \; \; \forall i=1,\ldots,n,
    \end{align*}
    and $e_{ij}$ are matrix units and non-trivial commutators given by
    \begin{align*}
    	&[x,B_i]=B_i, \; [y,A_i]=A_i, \; \; \forall i=1,\ldots,n,\\
    	&[E_i,B_k]=B_{k-i}, \; \; 1 \leq i < k \leq n,\\
    	&[E_i,A_k]=-A_{i+k}, \; \; 1 \leq i \leq n-1,\; 1 \leq k \leq n-i.
    \end{align*}

\begin{rem}\label{changebasis}
	With the change of basis 
	$$
	\left\{e_1,\ldots,e_n,f_1, \ldots,f_n,z\right\} \mapsto \left\{e_1,f_1,\ldots,e_n,f_n,z\right\},
	$$
	a derivation of $\mathfrak{l}_{2n+1}^{J_a}$ is represented by the $(2n+1) \times (2n+1)$ matrix
    \begin{equation*}\label{matrix1}
    	\begin{pmatrix}[cccccc|c]
    		M_1 & -\widetilde{M_2} & -\widetilde{M_3} & -\widetilde{M_4} & \cdots & -\widetilde{M}_{n} & 0 \\
    		M_2 & M_1 & -\widetilde{M_2} & -\widetilde{M_3} & \cdots & -\widetilde{M}_{n-1} & 0\\
    		M_3 & M_2 & M_1 & -\widetilde{M_2} & \cdots & -\widetilde{M}_{n-2} & 0\\
    		M_4 & M_3 & M_2 & M_1 & \cdots & -\widetilde{M}_{n-3} & 0\\
    		\vdots & \vdots & \vdots & \vdots &\ddots&\vdots & \vdots \\
    		M_{n} & M_{n-1} & M_{n-2} & M_{n-3} & \cdots & M_1 & 0 \\
    		\hline
    		v_1 & v_2 & v_3 & v_4 & \cdots & v_n & \operatorname{tr}(M_1)
    	\end{pmatrix}
    \end{equation*}
    where
    $$
    M_1=\begin{pmatrix}
    	\alpha_1 & 0 \\
    	0 & \beta_1 \\
    \end{pmatrix}, \; \;
    M_i=\begin{pmatrix}
    	0 & 0 \\0 & \alpha_i \\
    \end{pmatrix}, \; \;
    \widetilde{M}_i=\begin{pmatrix}
	\alpha_i & 0 \\0 & 0 \\
\end{pmatrix},
 \; \; \forall i=2,\ldots,n,
    $$
	$v_k=(\mu_k,\nu_k)$, for any $k=1,\ldots,n$ and $\operatorname{tr}(M_1)=\alpha_1+\beta_1$.\\
	
	In this case $\operatorname{Der}(\mathfrak{l}_{2n+1}^{J_a})$ has basis
	$$
	\lbrace x,y,E_1,\ldots,E_{n-1},A_1,\ldots,A_n,B_1,\ldots,B_n \rbrace,
	$$
	where
	\begin{align*}
	&x=\sum_{k=1}^{n} e_{2k-1,2k-1}+e_{2n+1,2n+1}, \; \; y=\sum_{k=1}^{n} e_{2k,2k}+e_{2n+1,2n+1},\\
	&E_i=\sum_{k=0}^{n-i+1}(e_{2(k+i+1),2(k+1)}-e_{2k+1,2(k+i)+1}), \; \; \forall i=1,\ldots,n-1,\\
	&A_i=e_{2n+1,2i-1}, \; \; B_i=e_{2n+1,2i}, \; \; \forall i=1,\ldots,n
	\end{align*}
	and the Lie brackets are given by
	\begin{align*}
	&[x,B_i]=B_i, \; [y,A_i]=A_i, \; \; \forall i=1,\ldots,n,\\
	&[E_i,B_k]=-B_{k-2i}, \; \; k > 2i,\\
	&[E_i,A_k]=A_{k+2i}, \; \; k+2i \leq 2n.
	\end{align*}
With this representation, one can check that
$$
\operatorname{Der}(\mathfrak{l}_{2n+1}^{J_a}) \subseteq \operatorname{Der}(\mathfrak{h}_{2n+1}),
$$
where the derivations of the $(2n+1)-$dimensional Heisenberg Lie algebra $\mathfrak{h}_{2n+1}$ were determined in \cite{dath:hal-01215682}, with respect to the symplectic basis $\lbrace e_1,f_1,\ldots,e_n,f_n,z\rbrace $ of $\mathfrak{h}_{2n+1}$. Later we present the derivations of the Heisenberg Leibniz algebra $\mathfrak{l}_{2n+1}^{J_0}$ and the ones of the Kronecker Leibniz algebra $\mathfrak{k}_n$ with respect to this basis, in order to compare them with the corresponding ones of the Heisenberg Lie algebra $\mathfrak{h}_{2n+1}$.
\end{rem}

	The commutator ideal of $\operatorname{Der}(\mathfrak{l}_{2n+1}^{J_a})$ is the abelian algebra of dimension $2n$ with basis
	$$
	\lbrace A_1,\ldots A_n, B_1,\ldots, B_n \rbrace,
	$$
	thus $\operatorname{Der}(\mathfrak{l}_{2n+1}^{J_a})$ is a two-step solvable Lie algebra. Moreover the lower central series is 
	$$
	\operatorname{Der}(\mathfrak{l}_{2n+1}^{J_a}) \supseteq \langle A_1,\ldots A_n, B_1,\ldots, B_n \rangle \supseteq \langle A_1,\ldots A_n, B_1,\ldots, B_n \rangle \supseteq ....
	$$
	so $\operatorname{Der}(\mathfrak{l}_{2n+1}^{J_a})$ is not nilpotent and its nilradical is the ideal
	$$
	N=\langle E_1,\ldots,E_{n-1},A_1,\ldots,A_n,B_1,\ldots,B_n \rangle.
	$$
	Finally the center $\operatorname{Z}(\operatorname{Der}(\mathfrak{l}_{2n+1}^{J_a}))$ is trivial and the algebra of inner derivations of $\mathfrak{l}_{2n+1}^{J_a}$ is
	$$
	\operatorname{Inn}(\mathfrak{l}_{2n+1}^{J_a})=\langle A_h,\ldots,A_n,B_1,\ldots,B_k \rangle,
	$$
	with $h=1$ and $k=n$ if $a \neq \pm 1$; $h=2$ and $k=n$ if $a=1$; and $h=1$ and $k=n-1$ if $a=-1$. Indeed
	$$
	\operatorname{ad}_{e_i}=B_{i-1}+(1+a)B_i, \; \; \forall i=2,\ldots,n,
	$$
	$$
	\operatorname{ad}_{f_j}=A_{j+1}+(-1+a)A_j, \; \; \forall j=1,\ldots,n-1
	$$
	and
	$$
	\operatorname{ad}_{e_1}=(1+a)B_1, \; \operatorname{ad}_{f_n}=(-1+a)A_n,
	$$
	thus for $a=1$ we have
	$$
	\operatorname{ad}_{f_n}=0, \; \operatorname{ad}_{f_j}=A_{j+1}, \; \; \forall j=1,\ldots,n-1
	$$
	and the matrix $A_1$ does not represent an inner derivations. In the same way, if $a=-1$, then $B_{n} \not\in \operatorname{Inn}(\mathfrak{l}_{2n+1}^{J_1})$. We will show later that
	$$
	\operatorname{AIDer}(\mathfrak{l}_{2n+1}^{J_a})=\langle A_1,\ldots,A_n,B_1,\ldots,B_n \rangle
	$$
	for every $a \in \mathbb{C}$. \\
	
When $a = 0$, a derivation of the Heisenberg Leibniz algebra $\mathfrak{l}_{2n+1}^{J_0}$ has the form
$$
\begin{pmatrix}
	A & C & 0 \\
	B & D & 0 \\
	\mu & \nu & \gamma \\
\end{pmatrix}
$$
where $A$, $D$, $\mu$, $\nu$ and $\gamma$ are as above,
$$
B=\left(
\begin{array}{*9{c}}
	0 & 0 & 0 & 0 & \ldots & 0 & 0 & 0 & 0 \\
	0 & 0 & 0 & 0 & \ldots & 0 & 0 & 0 & -b_{n+2}  \\
	0 &  0 & 0 & 0 & \ldots & 0 & 0 & b_{n+2} & 0 \\
	0 & 0 & 0 & 0 & \ldots & 0 & -b_{n+2} & 0 & -b_{n+4}\\
	\vdots &  \vdots &  \vdots & \vdots &\iddots & \vdots & \vdots & \vdots & \vdots \\
	0 & 0 & 0 &0 & \ldots & b_{2n-6} & 0 & b_{2n-4} & 0\\
	0 & 0 & 0 & -b_{n+2} & \ldots & 0 & -b_{2n-4} & 0 & -b_{2n-2}\\
	0 &  0 & b_{n+2} & 0 &  \ldots & b_{2n-4} & 0 & b_{2n-2} & 0 \\
	0 & -b_{n+2} & 0 & -b_{n+4} &  \ldots & 0 & -b_{2n-2} & 0 & -b_{2n} \\
\end{array}\right), \; \;
$$
$$
C=\left(
\begin{array}{*9{c}}
     c_2 & 0 & c_4 & 0 & \ldots & c_{n-2} & 0 & c_{n} & 0 \\
	0 & -c_4 & 0 & -c_6 & \ldots & 0 & -c_n & 0 &0  \\
	 c_4 & 0 & c_6 & 0 & \ldots & c_n & 0 & 0 & 0\\
	0 & -c_6 & 0 & -c_8  & \ldots & 0 & 0 & 0 & 0\\
	\vdots &  \vdots &  \vdots & \vdots &\iddots & \vdots & \vdots & \vdots & \vdots \\
	c_{n-1} & 0 & c_n & 0 & \ldots & 0 & 0 & 0 & 0\\
	0 & -c_n & 0 & 0 & \ldots & 0 & 0 & 0 & 0\\
	c_n &  0 &  0 & 0 &  \ldots & 0 & 0 & 0 & 0 \\
	0 &  0 & 0 & 0 &  \ldots & 0 & 0 & 0 & 0 \\
\end{array}\right), \; \;
$$
if $n$ is even,
$$
B=\left(
\begin{array}{*9{c}}
	0 & 0 & 0 & 0 & \ldots & 0 & 0 & 0 & b_{n+1} \\
	0 & 0 & 0 & 0 & \ldots & 0 & 0 & -b_{n+1} & 0  \\
	0 &  0 & 0 & 0 & \ldots & 0 & b_{n+1} & 0 & b_{n+3}\\
	0 & 0 & 0 & 0 & \ldots & -b_{n+1} & 0 & -b_{n+3} & 0\\
	\vdots &  \vdots &  \vdots & \vdots &\iddots & \vdots & \vdots & \vdots & \vdots \\
	0 & 0 & 0 & -b_{n+1} & \ldots & -b_{2n-3} & 0 & -b_{2n-2} & 0\\
	0 & 0 & b_{n+1} & 0 & \ldots & 0 & b_{2n-3} & 0 & b_{2n-2}\\
	0 &  -b_{n+1} & 0 & -b_{n+3}&  \ldots & -b_{2n-2} & 0 & -b_{2n} & 0 \\
	b_{n+1} & 0 & b_{n+3} & 0 &  \ldots & 0 & b_{2n-2} & 0 & b_{2n} \\
\end{array}\right), \; \;
$$
$$
C=\left(
\begin{array}{*9{c}}
	c_2 & 0 & c_4 & 0 & \ldots & 0 & c_{n-1} & 0 & c_{n+1} \\
	0 & -c_4 & 0 & -c_6 & \ldots & -c_{n-1} & 0 & -c_{n+1} &0  \\
	c_4 & 0 & c_6 & 0 & \ldots & 0 & c_{n+1} & 0 & 0\\
	0 & -c_6 & 0 & -c_8 & \ldots & -c_{n+1} & 0 & 0 & 0\\
	\vdots &  \vdots &  \vdots & \vdots &\iddots & \vdots & \vdots & \vdots & \vdots \\
	0 & -c_{n-1} & 0 & -c_{n+1} & \ldots & 0 & 0 & 0 & 0\\
	c_{n-1} & 0 & c_{n+1} & 0 & \ldots & 0 & 0 & 0 & 0\\
	0 &  -c_{n+1} &  0 & 0 &  \ldots & 0 & 0 & 0 & 0 \\
	c_{n+1} &  0 & 0 & 0 &  \ldots & 0 & 0 & 0 & 0 \\
\end{array}\right), \; \;
$$
if $n$ is odd.\\

If we reorder the basis basis as in Remark \ref{changebasis}, then a derivation of $\mathfrak{l}_{2n+1}^{J_0}$ is represented by
	\begin{equation*}\label{matrix2}
		\setcounter{MaxMatrixCols}{12}
		\begin{pmatrix}[ccccccccccc|c]
			\alpha_1 & c_2 & -\alpha_2 & 0 & -\alpha_3 & c_4 &  \cdots & -\alpha_{n-1} & c_n & -\alpha_{n} & 0 & 0 \\
			0 & \beta_1 & 0 & 0 & 0 & 0 & \cdots & 0 & 0 & 0 & 0 & 0\\
			0 & 0 & \alpha_1 & -c_4 & -\alpha_2 & 0 & \cdots & -\alpha_{n-2} & 0 & -\alpha_{n-1} & 0 & 0 \\
			0 & \alpha_2 & 0 & \beta_1 & 0 & 0 & \cdots & 0 & 0 & -b_{n+2} & 0 & 0 \\
			0 & c_4 & 0 & 0 & \alpha_1 & c_6 & \cdots & -\alpha_{n-3} & 0 & -\alpha_{n-2} & 0 & 0 \\
			0 & \alpha_3 & 0 & \alpha_2  & 0 & \beta_1 & \cdots & b_{n+2} & 0 & 0 & 0 & 0\\
			\vdots & \vdots & \vdots & \vdots & \vdots & \vdots & \ddots & \vdots & \vdots & \vdots & \vdots & \vdots \\
			0 & c_n & 0 & 0 & 0 & 0 & \cdots & \alpha_1 & 0 & -\alpha_2 & 0 & 0 \\
			0 & \alpha_{n-1} & 0 & \alpha_{n-2} & b_{n+2} & \alpha_{n-3} & \cdots & b_{2n-2} & \beta_1 & 0 & 0 & 0 \\
			0 & 0 & 0 & 0 & 0 & 0 & \cdots & 0 & 0 & \alpha_1 & 0 & 0 \\
			0 & \alpha_{n} & -b_{n+2} & \alpha_{n-1} & 0 & \alpha_{n-2}  & \cdots & 0 & \alpha_2 & b_{2n} & \beta_1 & 0 \\
			\hline
			\mu_1 & \nu_1 & \mu_2 & \nu_2 & \mu_3 & \nu_3 &\cdots & \mu_{n-1} & \nu_{n-1} & \mu_n & \nu_n & \alpha_1+\beta_1
		\end{pmatrix}
	\end{equation*}
	if $n$ is even, and 
	\begin{equation*}\label{matrix3}
		\setcounter{MaxMatrixCols}{12}
		\begin{pmatrix}[ccccccccccc|c]
			 \alpha_1 & c_2 & -\alpha_2 & 0 & -\alpha_3 & c_4 &  \cdots & -\alpha_{n-1} & 0 & -\alpha_{n} & c_{n+1} & 0 \\
			0 & \beta_1 & 0 & 0 & 0 & 0 & \cdots & 0 & 0 & b_{n+1} & 0 & 0 \\
			0 & 0 &  \alpha_1 & -c_4 & -\alpha_2 & 0 & \cdots & -\alpha_{n-2} & -c_{n+1} &  -\alpha_{n-1} &  0 & 0 \\
			0 & \alpha_2 & 0 & \beta_1 & 0 & 0 & \cdots & -b_{n+1} & 0 & 0  & 0 & 0 \\
			0 & c_4 & 0 & 0 &  \alpha_1 & c_6 & \cdots & -\alpha_{n-3} & 0 &  -\alpha_{n-2}  & 0 & 0 \\
			0 & \alpha_3 & 0 & \alpha_2  & 0 & \beta_1 & \cdots & 0 & 0& b_{n+3}  & 0 & 0 \\
			\vdots & \vdots & \vdots & \vdots & \vdots & \vdots & \ddots & \vdots & \vdots & \vdots & \vdots \\
			0 & 0 & 0 & -c_{n+1} & 0 & 0 & \cdots &  \alpha_1 & 0 & -\alpha_2 & 0 & 0 \\
			0 & \alpha_{n-1} & -b_{n+1} & \alpha_{n-2} & 0 & \alpha_{n-3} & \cdots & -b_{2n-2}  & \beta_1 & 0 & 0 & 0 \\
			0 & c_{n+1} & 0 & 0 & 0 & 0 & \cdots & 0  & 0 &  \alpha_1 & 0 & 0 \\
			b_{n+1} & \alpha_{n} & 0 & \alpha_{n-1}  & b_{n+3} & \alpha_{n-2} & \cdots & 0 & \alpha_1 & b_{2n} & \beta_1 & 0 \\
			\hline
			\mu_1 & \nu_1 & \mu_2 & \nu_2 & \mu_3 & \nu_3 &\cdots & \mu_{n-1}  & \nu_{n-1} & \mu_n & \nu_n &  \alpha_1+\beta_1
		\end{pmatrix}
	\end{equation*}
if $n$ is odd. We can now study in details these two cases.\\
  
  If $n$ is even, then $\operatorname{Der}(\mathfrak{l}_{2n+1}^{J_0})$ is a Lie algebra of dimension $4n+1$ with basis
		$$
		\lbrace x,y,E_1,\ldots,E_{n-1},c_2,c_4,\ldots,c_n,b_{n+2},b_{n+4}\ldots,b_{2n},A_1,\ldots,A_n,B_1,\ldots,B_n \rbrace,
		$$
		where
		\begin{align*}
		&x=\sum_{k=1}^{n} e_{2k-1,2k-1}+e_{2k+1,2k+1}, \; \; \\
		&y=\sum_{k=1}^{n} e_{2k,2k}+e_{2k+1,2k+1},\\
		&E_i=\sum_{k=0}^{n-i+1}(e_{2(k+i+1),2(k+1)}-e_{2k+1,2(k+i)+1}), \; \; \forall i=1,\ldots,n-1,\\
		&A_i=e_{2n+1,2i-1}, \; \; B_i=e_{2n+1,2i}, \; \; \forall i=1,\ldots,n,\\
		&c_h=\sum_{i=0}^{h-2}(-1)^i e_{2(h-i-1)-1,2(1+i)}, \; \; \forall h=2,4,\ldots,n,\\
		&b_h=\sum_{i=0}^{2n-h}(-1)^i e_{2(n-i),2(h-n+i)-1}, \; \; \forall h=n+2,n+4,\ldots,2n
	\end{align*}
		and commutators
		\begin{align*}
		&[x,B_i]=B_i, \; [y,A_i]=A_i, \; \; \forall i=1,\ldots,n,\\
		&[E_i,B_k]=-B_{k-2i}, \; \; k > 2i,\\
		&[E_i,A_k]=A_{k+2i}, \; \; k+2i \leq 2n,\\
		&[x,c_{h}]=c_{h}, \; [y,c_{h}]=-c_h, \; \; \forall h=2,4,\ldots,n,\\
		&[x,b_{h}]=-b_{h}, \; [y,b_{h}]=b_h, \; \; \forall h=n+2,n+4,\ldots,2n,\\
		&[A_i,c_k]=(-1)^{i+1}B_{k-i}, \; [B_i,b_k]=(-1)^iA_{k-i}, \; \; 1 \leq k-i \leq n,\\
		&[c_k,b_h]=E_{h-k}, \; \; h-k \geq 1,\\
		&[\alpha_{2i},c_h]=-2c_{h-2i}, \; \; h-2i >0,\\
		&[\alpha_{2i},b_h]=2b_{h+2i}, \; \; h+2i \leq 2n.
		\end{align*}
		
		Then the commutator ideal of $\operatorname{Der}(\mathfrak{l}_{2n+1}^{J_0})$ has basis
		$$
		\lbrace E_2,E_4,\ldots,E_{n-2},c_2,c_4,\ldots,c_n,b_{n+2},b_{n+4}\ldots,b_{2n},A_1,\ldots,A_n,B_1,\ldots,B_n \rbrace 
		$$
		and we have a $(\frac{n}{2}+1)-$step solvable Lie algebra with derived series
		$$
		\operatorname{Der}(\mathfrak{l}_{2n+1}^{J_0}) \supseteq [\operatorname{Der}(\mathfrak{l}_{2n+1}^{J_0}),\operatorname{Der}(\mathfrak{l}_{2n+1}^{J_0})] \supseteq
		$$
		$$
		\supseteq \langle E_2,E_4,\ldots,E_{n-2},c_2,c_4,\ldots,c_{n-2},b_{n+4},\ldots,b_{2n},A_2,\ldots,A_n,B_1,\ldots,B_{n-1} \rangle \supseteq \ldots
		$$
		$$
		\ldots \supseteq \langle c_2,b_{2n},A_{\frac{n}{2}+1},\ldots,A_{n},B_1,\ldots,B_{\frac{n}{2}} \rangle \supseteq 0
		$$
		Moreover, $\operatorname{Der}(\mathfrak{l}_{2n+1}^{J_0})$ is not nilpotent and its nilradical is the ideal
		$$
		N=\langle E_1,\ldots,E_{n-1},c_2,c_4,\ldots,c_n,b_{n+2},b_{n+4}\ldots,b_{2n},A_1,\ldots,A_n,B_1,\ldots,B_n \rangle.
		$$

	If $n$ is odd, then the algebra of derivations of $\mathfrak{l}_{2n+1}^{J_0}$ has dimension $4n+2$ and it is generated by
		$$
		\lbrace x,y,E_1,\ldots,E_{n-1},c_2,c_4,\ldots,c_{n+1},b_{n+1},b_{n+3}\ldots,b_{2n},A_1,\ldots,A_n,B_1,\ldots,B_n \rbrace.
		$$
		The Lie brackets are the same of the ones listed for the case $n$ even, except for the facts that $[B_i,b_k]=(-1)^{i+1}A_{k-i}$, for any $1 \leq k-i \leq n$, and $[c_{n+1},b_{n+1}]=x-y$. Then the commutator ideal is the subspace generated by
		$$
		\lbrace x-y,E_2,E_4,\ldots,E_{n-1},c_2,c_4,\ldots,c_{n+1},b_{n+1},b_{n+3}\ldots,b_{2n},A_1,\ldots,A_n,B_1,\ldots,B_n \rbrace. 
		$$
		In this case the Lie algebra of derivations is not solvable since
		$$
		[[\operatorname{Der}(\mathfrak{l}_{2n+1}^{J_0}),\operatorname{Der}(\mathfrak{l}_{2n+1}^{J_0})],[\operatorname{Der}(\mathfrak{l}_{2n+1}^{J_0}),\operatorname{Der}(\mathfrak{l}_{2n+1}^{J_0})]]=[\operatorname{Der}(\mathfrak{l}_{2n+1}^{J_0}),\operatorname{Der}(\mathfrak{l}_{2n+1}^{J_0})]
		$$
		and the Levi decomposition is given by
		$$
		\operatorname{Der}(\mathfrak{l}_{2n+1}^{J_0})=R \rtimes S,
		$$
		where the radical of the Lie algebra is
		$$
		R=\langle x+y,E_1,\ldots,E_{n-1},c_2,c_4,\ldots,c_{n-1},b_{n+3},b_{n+5}\ldots,b_{2n},A_1,\ldots,A_n,B_1,\ldots,B_n \rangle
		$$
		and the Levi complement is
		$$
		S=\langle x-y,c_{n+1},b_{n+1} \rangle.
		$$
		Finally the nilradical is the ideal
		$$
		N=\langle E_1,\ldots,E_{n-1},c_2,c_4,\ldots,c_{n-1},b_{n+3},b_{n+5}\ldots,b_{2n},A_1,\ldots,A_n,B_1,\ldots,B_n \rangle.
		$$

     In both cases $n$ is even or odd, we have that $\operatorname{Z}(\operatorname{Der}(\mathfrak{l}_{2n+1}^{J_0}))=0$ and the Lie algebra of inner derivations is represented by the matrices of the type
     \begin{equation*}
     	$$\[
     	\left(
     	\begin{array}{c|c}
     		& 0 \\
     		\textbf{\Large0}   & \vdots \\
     		& 0 \\
     		\hline
     		\mu_1 \; \nu_1 \ldots \mu_{n} \; \nu_n & 0 \\
     	\end{array}
     	\right),
     	\]$$
     \end{equation*}
     thus $\operatorname{Inn}(\mathfrak{l}_{2n+1}^{J_0})$ is an abelian algebra of dimension $2n$. Moreover, for every $n \in \mathbb{N}$ and for every $a \in \mathbb{C}^{\ast}$, we observe that
     $$
     \operatorname{Der}(\mathfrak{h}_{2n+1}) \supseteq 
     \operatorname{Der}(\mathfrak{l}_{2n+1}^{J_0}) \supseteq \operatorname{Der}(\mathfrak{l}_{2n+1}^{J_a}).
     $$

	\subsection{The real case}
	
	Irreducible polynomials in $\R[x]$ have degree one or two. Let $f(x)\in\R[x]$ be an irreducible monic polynomial. If $f(x)=x-a$, then we obtain the same results of the complex case. So we suppose that $f(x)=x^2+C x+D$, with $C^2-4D<0$.\\
	
	Let $z=a + i b \in \mathbb{C} \setminus \mathbb{R}$ be a root of $f(x)$. Then $f(x)=(x-z)(x-\bar{z})$ and the companion matrix $A$ of $f(x)^n$ in similar to the $2n \times 2n$ real block matrix
	$$
	J_R=\begin{pmatrix}
		R&0&\cdots&0\\I_2&R&\cdots&0\\\vdots&\ddots&\ddots&\vdots\\0&\cdots&I_2&R
	\end{pmatrix},
	$$
	where
	$$
	R=\begin{pmatrix}
		a & b \\
		-b & a
	\end{pmatrix}
	$$
	is the realification of the complex number $z$. Thus $\mathfrak{l}_{4n+1}^{A} \cong \mathfrak{l}_{4n+1}^{J_R}$ 
	and $\mathfrak{l}_{4n+1}^{J_R}$ is the realification of the complex algebra $\mathfrak{l}_{2n+1}^{J_z}$. In \cite{Falcone2017class} the derivations of the realification of the $(2n+1)-$dimensional Heisenberg Lie algebra $\mathfrak{h}_{2n+1}$ were studied. We want to find the conditions such that the realification of a derivation of the complex algebra $\mathfrak{l}_{2n+1}^{J_z}$, with $z=a+ib \in \mathbb{C} \setminus \mathbb{R}$, is a derivation of the real algebra $\mathfrak{l}_{4n+1}^{J_R}$. We will investigate the case $n=1$.\\
	
	Let $\lbrace e_1,f_1,e_2,f_2,z \rbrace$ be a basis of the real algebra $\mathfrak{l}_5^R$. The non-trivial commutators are
	\begin{align*}
		&[e_i,f_i]=(1+a)z,\; [f_i,e_i]=(-1+a)z, \; \; \forall i=1,2,\\
		&[e_1,f_2]=[f_2,e_1]=bz,\; [e_2,f_1]=[f_1,e_2]=-bz
	\end{align*}
	and it comes out that a general derivation of $\mathfrak{l}_{5}^{R}$ is represented by the matrix 
	$$
	\begin{pmatrix}[cccc|c]
		\alpha_1 & 0 & \alpha_2 & 0 & 0\\
		0 & \beta_1 & 0 & \alpha_2 & 0 \\
		-\alpha_2 & 0 & \alpha_1 & 0 & 0 \\
		0 & -\alpha_2 & 0 & \beta_1 & 0 \\
		\hline
		\mu_1 & \nu_1 & \mu_2 & \nu_2 & \alpha_1+\beta_1 \\
	\end{pmatrix}
	$$
	if $a \neq 0$ and by 
		$$
	\begin{pmatrix}[cccc|c]
		 \alpha_1 & \delta & \alpha_2 & 0 & 0\\
		\delta' & \beta_1 & 0 & \alpha_2 & 0 \\
		-\alpha_2 & 0 &  \alpha_1 & \delta & 0 \\
		0 & -\alpha_2 & \delta' & \beta_1 & 0 \\
		\hline
		\mu_1 & \nu_1 & \mu_2 & \nu_2 & \alpha_1+\beta_1 \\
	\end{pmatrix}
	$$
	if $a=0$. Then
	\begin{itemize}
		\item if $a \neq 0$, $\operatorname{Der}(\mathfrak{l}_{5}^{R})$ is generated by the set
		$$
		\lbrace x,y,E,A_1,A_2,B_1,B_2 \rbrace,
		$$
		where $x=e_{11}+e_{33}+e_{55}$, $y=e_{22}+e_{44}+e_{55}$, $E=e_{13}+e_{24}-e_{31}-e_{42}$, $A_i=e_{5,2i-1}$ and $B_i=e_{5,2i}$, for every $i=1,2$, and the non-trivial Lie brackets are
		\begin{align*}
			&[x,B_i]=B_i, \; [y,A_i]=A_i, \; \; \forall i=1,2,\\
			&[E,A_1]=-A_2, \; [E,A_2]=-A_1,\\
			&[E,B_1]=-B_2, \; [E,B_2]=-B_1.
		\end{align*}
		
		Then we have a solvable Lie algebra with abelian commutator ideal generated by
		$$
		\lbrace A_1,A_2,B_1,B_2 \rbrace,
		$$
		which coincides with the ideal $\operatorname{Inn}(\mathfrak{l}_5^R)$ and with the nilradical of the Lie algebra itself. Moreover the center $\operatorname{Z}(\operatorname{Der}(\mathfrak{l}_{5}^{R}))$ is trivial.
		\item if $a = 0$, a basis of $\operatorname{Der}(\mathfrak{l}_{5}^{R})$ is
		$$
		\lbrace x,y,E,F,G,A_1,A_2,B_1,B_2 \rbrace,
		$$
		where $x,y,E,A_1,A_2,B_1,B_2$ are defined as above, $F=e_{12}+e_{34}$, $G=e_{21}+e_{43}$ and the non-trivial Lie brackets are given by the ones above and by
		\begin{align*}
			&[x,F]=F, \; [x,G]=-G, \;
			[y,F]=-F,\;[y,G]=-G,\\
			&[F,G]=x-y, \; [F,A_i]=-B_i, \; [F,B_i]=-A_i, \; \, \forall i=1,2.
		\end{align*}
		It follows that $\operatorname{Z}(\operatorname{Der}(\mathfrak{l}_{5}^{R}))=0$ and the Lie algebra is not solvable. Its radical is given by the ideal
		$$
		R=\langle x+y,E,A_1,A_2,B_1,B_2 \rangle,
		$$
		a Levi complement is the semisimple Lie algebra
		$$
		S=\langle x-y,F,G \rangle
		$$
		and the nilradical of $\operatorname{Der}(\mathfrak{l}_{5}^{R})$ is the abelian four-dimensional algebra
		$$
		N=\langle A_1,A_2,B_1,B_2 \rangle \cong \mathbb{R}^4
		$$
		and again it coincides with the set of inner derivations of $\mathfrak{l}_{5}^{R}$.
		\end{itemize}

	Now let
	$$
	\begin{pmatrix}
		\alpha & 0 & 0 \\
		0 & \beta & 0 \\
		\mu & \nu & \alpha+\beta \\
	\end{pmatrix} 
	$$
	be a derivation of the complex Heinseberg algebra $\mathfrak{l}_3^z$, where $z=a+ib$. Then its realification is represented by the matrix
	$$
	\begin{pmatrix}[cccc|c]
		\Re(\alpha) & \Im(\alpha) & 0 & 0 & 0\\
		-\Im(\alpha) & \Re(\alpha) & 0 & 0 & 0 \\
		0 & 0 & \Re(\beta) & \Im(\beta) & 0 \\
		0 & 0 & -\Im(\beta) & \Re(\beta) & 0 \\
		\hline
		\mu_1 & \nu_1 & \mu_2 & \nu_2 & \gamma \\
	\end{pmatrix}
	$$
	and this is a derivation of the real Heisenberg Leibniz algebra $\mathfrak{l}_{5}^{R}$ if and only if
	$$
	\alpha=\beta \in \mathbb{R}
	$$
	in both cases that $a \neq 0$ or $a=0$. Then the set of realifications of the derivations of $\mathfrak{l}_3^z$ that are derivations of the real algebra $\mathfrak{l}_{5}^{R}$ form the proper Lie subalgebra of the matrices of the form
	$$
\begin{pmatrix}[cccc|c]
	\alpha & 0 & 0 & 0 & 0\\
	0 & \alpha & 0 & 0 & 0 \\
	0 & 0 & \alpha & 0 & 0 \\
	0 & 0 & 0 & \alpha & 0 \\
	\hline
	\mu_1 & \nu_1 & \mu_2 & \nu_2 & 2\alpha \\
\end{pmatrix}
$$
	of $\operatorname{Der}(\mathfrak{l}_{5}^{R})$.
	
	\section{Derivations of the Kronecker Leibniz algebra $\mathfrak{k}_n$}
	
	Now we return to the case that $\F$ is a field with $\operatorname{char}(\F)\neq 2$. Let $n \in \mathbb{N}$ and let $\mathfrak{k}_n$ be the \emph{Kronecker Leibniz algebra}. We fix the basis  $\lbrace e_1,\ldots,e_n,f_1,\ldots,f_n,z\rbrace$ of $\mathfrak{k}_n$.
	
	\begin{thm}
		A linear endomorphism $d\colon \mathfrak{k}_n \rightarrow \mathfrak{k}_n$ is a derivation of the Kronecker algebra $\mathfrak{k}_{n}$ if and only if it has the form
		$$
		\begin{pmatrix}
			A & C & 0 \\
			B & D & 0 \\
			\mu & \nu & \gamma \\
		\end{pmatrix}
		$$
		where
		$$
		A=\left(
		\begin{array}{*5{c}}
			\alpha_1 & \alpha_2 & \alpha_3& \ldots & \alpha_n \\
			0 & \alpha_1 & \alpha_2 & \ldots & \alpha_{n-1}  \\
			0 &  0 &\alpha_1 & \ldots & \alpha_{n-2}  \\
			\vdots &  \vdots &  \vdots & \ddots & \vdots \\
			0 &  0 &  0 &  \ldots & \alpha_1 \\
		\end{array}\right), \; \;
		D= \left(
		\begin{array}{*5{c}}
			\beta_1 & 0 & 0 & \ldots & 0 \\
			-\alpha_2 & \beta_1 & 0 & \ldots & 0  \\
			-\alpha_3 &  -\alpha_2 &\beta_1 & \ldots & 0  \\
			\vdots &  \vdots &  \vdots & \ddots & \vdots \\
			-\alpha_n &  -\alpha_{n-1} &  -\alpha_{n-2} &  \ldots & \beta_1 \\
		\end{array}\right),
		$$
		\\
		$$
		B=\left(
		\begin{array}{*9{c}}
			0 & 0 & 0 & 0 & \ldots & 0 & 0 & 0 & b_{n+1} \\
			0 & 0 & 0 & 0 & \ldots & 0 & 0 & -b_{n+1} & 0  \\
			0 &  0 & 0 & 0 & \ldots & 0 & b_{n+1} & 0 & b_{n+3}\\
			0 & 0 & 0 & 0 & \ldots & -b_{n+1} & 0 & -b_{n+3} & 0\\
			\vdots &  \vdots &  \vdots & \vdots &\iddots & \vdots & \vdots & \vdots & \vdots \\
			0 & 0 & 0 & b_{n+1} & \ldots & 0 & b_{2n-5} & 0 & b_{2n-3}\\
			0 & 0 & -b_{n+1} & 0 & \ldots & -b_{2n-5} & 0 & -b_{2n-3} & 0\\
			0 &  b_{n+1} & 0 & b_{n+3}&  \ldots & 0 & b_{2n-3} & 0 & b_{2n-1} \\
			 -b_{n+1} & 0 & -b_{n+3} & 0 &  \ldots & -b_{2n-3} & 0 & -b_{2n-1} & 0 \\
		\end{array}\right), \; \;
	$$
    $$
		C=\left(
		\begin{array}{*9{c}}
			0 & c_3 & 0 & c_5 & \ldots & 0 & c_{n-1} & 0 & c_{n+1} \\
			-c_3 & 0 & -c_5 & 0 & \ldots & -c_{n-1} & 0 & -c_{n+1} &0  \\
			0 &  c_5 & 0 & c_7 & \ldots & 0 & c_{n+1} & 0 & 0\\
			-c_5 & 0 & -c_7 & 0 & \ldots & -c_{n+1} & 0 & 0 & 0\\
			\vdots &  \vdots &  \vdots & \vdots &\iddots & \vdots & \vdots & \vdots & \vdots \\
			0 & c_{n-1} & 0 & c_{n+1} & \ldots & 0 & 0 & 0 & 0\\
			-c_{n-1} & 0 & -c_{n+1} & 0 & \ldots & 0 & 0 & 0 & 0\\
			0 &  c_{n+1} &  0 & 0 &  \ldots & 0 & 0 & 0 & 0 \\
			-c_{n+1} &  0 & 0 & 0 &  \ldots & 0 & 0 & 0 & 0 \\
		\end{array}\right), \; \;
		$$
		if $n$ is even,
		$$
		B=\left(
		\begin{array}{*9{c}}
			0 & 0 & 0 & 0 & \ldots & 0 & 0 & 0 & 0 \\
			0 & 0 & 0 & 0 & \ldots & 0 & 0 & 0 & -b_{n+2}  \\
			0 &  0 & 0 & 0 & \ldots & 0 & 0 & b_{n+2} & 0\\
			0 & 0 & 0 & 0 & \ldots & 0 & -b_{n+2} & 0 & -b_{n+4}\\
			\vdots &  \vdots &  \vdots & \vdots &\iddots & \vdots & \vdots & \vdots & \vdots \\
			0 & 0 & 0 & 0 & \ldots & 0 & -b_{2n-5} & 0 & -b_{2n-3}\\
			0 & 0 & 0 & b_{n+2} & \ldots & b_{2n-5} & 0 & b_{2n-3} & 0\\
			0 &  0 &  -b_{n+2} & 0 &  \ldots & 0 & -b_{2n-3} & 0 & -b_{2n-1} \\
			0 &  b_{n+2} & 0 & b_{n+4} &  \ldots & b_{2n-3} & 0 & b_{2n-1} & 0 \\
		\end{array}\right), \; \;
	$$
	$$
		C=\left(
		\begin{array}{*9{c}}
			0 & c_3 & 0 & c_5 & \ldots & c_{n-2} & 0 & c_{n} & 0 \\
			-c_3 & 0 & -c_5 & 0 & \ldots & 0 & -c_n & 0 &0  \\
			0 &  c_5 & 0 & c_7 & \ldots & c_n & 0 & 0 & 0\\
			-c_5 & 0 & -c_7 & 0 & \ldots & 0 & 0 & 0 & 0\\
			\vdots &  \vdots &  \vdots & \vdots &\iddots & \vdots & \vdots & \vdots & \vdots \\
			-c_{n-2} & 0 & -c_n & 0 & \ldots & 0 & 0 & 0 & 0\\
			0 & c_n & 0 & 0 & \ldots & 0 & 0 & 0 & 0\\
			-c_n &  0 &  0 & 0 &  \ldots & 0 & 0 & 0 & 0 \\
			0 &  0 & 0 & 0 &  \ldots & 0 & 0 & 0 & 0 \\
		\end{array}\right), \; \;
		$$
		if $n$ is odd and
		$$
		\mu=(\mu_1,\mu_2,\mu_3,\ldots,\mu_n), \; \; \nu=(\nu_1,\nu_2,\nu_3,\ldots,\nu_n), \; \; \gamma=\alpha_1+\beta_1.
		$$ 
	\end{thm}

\begin{rem}
	Again, with the change of basis 
	$$
	\left\{e_1,\ldots,e_n,f_1, \ldots,f_n,z\right\} \mapsto \left\{e_1,f_1,\ldots,e_n,f_n,z\right\},
	$$
	a derivation of $\mathfrak{k}_{n}$ is represented by the $(2n+1) \times (2n+1)$ matrix
	\begin{equation*}
		\setcounter{MaxMatrixCols}{12}
		\begin{pmatrix}[ccccccccccc|c]
			\alpha_1 & 0 & -\alpha_2 & c_3 & -\alpha_3 & 0 &  \cdots & -\alpha_{n-1} & 0 & -\alpha_{n} & c_{n+1} & 0 \\
			0 & \beta_1 & 0 & 0 & 0 & 0 & \cdots & 0 & 0 & b_{n+1} & 0 & 0 \\
			0 & -c_3 & \alpha_1 & 0 & -\alpha_2 & -c_5 & \cdots & -\alpha_{n-2} & -c_{n+1} &  -\alpha_{n-1} &  0 & 0 \\
			0 & \alpha_2 & 0 & \beta_1 & 0 & 0 & \cdots & -b_{n+1} & 0 & 0  & 0 & 0 \\
			0 & 0 & 0 & c_5 & \alpha_1 & 0 & \cdots & -\alpha_{n-3} & 0 &  -\alpha_{n-2}  & 0 & 0 \\
			0 & \alpha_1 & 0 & \alpha_2  & 0 & \beta_1 & \cdots & 0 & 0& -b_{n+3}  & 0 & 0 \\
			\vdots & \vdots & \vdots & \vdots & \vdots & \vdots & \ddots & \vdots & \vdots & \vdots & \vdots & \vdots \\
			0 & 0 & 0 & c_{n+1} & 0 & 0 & \cdots & \alpha_1 & 0 & -\alpha_2 & 0 & 0 \\
			0 & \alpha_{n-1} & b_{n+1} & \alpha_{n-2} & 0 & \alpha_{n-3} & \cdots & 0  & \beta_1 & -b_{2n-1} & 0 & 0 \\
			0 & -c_{n+1} & 0 & 0 & 0 & 0 & \cdots & 0  & 0 & \alpha_1 & 0 & 0 \\
			-b_{n+1} & \alpha_{n} & 0 & \alpha_{n-1}  & b_{n+3} & \alpha_{n-2} & \cdots & b_{2n-1} & \alpha_2 & 0 & \beta_1 & 0 \\
			\hline 
			\mu_1 & \nu_1 & \mu_2 & \nu_2 & \mu_3 & \nu_3 &\cdots & \mu_{n-1}  & \nu_{n-1} & \mu_n & \nu_n & \alpha_1+\beta_1
		\end{pmatrix}
	\end{equation*}
	if $n$ is even and
	\begin{equation*}
		\setcounter{MaxMatrixCols}{12}
		\begin{pmatrix}[ccccccccccc|c]
			\alpha_1 & 0 & -\alpha_2 & c_3 & -\alpha_3 & 0 &  \cdots & -\alpha_{n-1} & c_n & -\alpha_{n} & 0 & 0 \\
			0 & \beta_1 & 0 & 0 & 0 & 0 & \cdots & 0 & 0 & 0 & 0 & 0\\
			0 & -c_3 & \alpha_1 & 0 & -\alpha_2 & -c_5 & \cdots & -\alpha_{n-2} & 0 & -\alpha_{n-1} & 0 & 0 \\
			0 & \alpha_2 & 0 & \beta_1 & 0 & 0 & \cdots & 0 & 0 & -b_{n+2} & 0 & 0 \\
			0 & 0 & 0 & -c_5 & \alpha_1 & 0 & \cdots & -\alpha_{n-3} & 0 & -\alpha_{n-2} & 0 & 0 \\
			0 & \alpha_3 & 0 & \alpha_2  & 0 & \beta_1 & \cdots & b_{n+2} & 0 & 0 & 0 & 0\\
			\vdots & \vdots & \vdots & \vdots & \vdots & \vdots & \ddots & \vdots & \vdots & \vdots & \vdots & \vdots \\
			0 & c_n & 0 & 0 & 0 & 0 & \cdots & \alpha_1 & 0 & -\alpha_2 & 0 & 0 \\
			0 & \alpha_{n-1} & 0 & \alpha_{n-2} & -b_{n+2} & \alpha_{n-3} & \cdots & 0 & \beta_1 & -b_{2n-1} & 0 & 0 \\
			0 & 0 & 0 & 0 & 0 & 0 & \cdots & 0 & 0 & \alpha_1 & 0 & 0 \\
			0 & \alpha_{n} & b_{n+2} & \alpha_{n-1} & 0 & \alpha_{n-2}  & \cdots & b_{2n-1} & \alpha_2 & 0 & \beta_1 & 0 \\
			\hline
			\mu_1 & \nu_1 & \mu_2 & \nu_2 & \mu_3 & \nu_3 &\cdots & \mu_{n-1}  & \nu_{n-1} & \mu_n & \nu_n & \alpha_1+\beta_1
		\end{pmatrix}
	\end{equation*}
	if $n$ is odd.
\end{rem}
	
 We can now describe the main properties of the Lie algebra $\operatorname{Der}(\mathfrak{k}_n)$.
	\begin{itemize}
		\item If $n$ is even, then $\operatorname{Der}(\mathfrak{k}_{n})$ has basis
		$$
		\lbrace x,y,E_1,\ldots,E_{n-1},c_3,c_5,\ldots,c_{n+1},b_{n+1},b_{n+3}\ldots,b_{2n-1},A_1,\ldots,A_n,B_1,\ldots,B_n \rbrace,
		$$
		where
		\begin{align*}
		&x=\sum_{k=1}^{n} e_{2k-1,2k-1}+e_{2k+1,2k+1}, \; \; \\
		&y=\sum_{k=1}^{n} e_{2k,2k}+e_{2k+1,2k+1},\\
		&E_i=\sum_{k=0}^{n-i+1}(e_{2(k+i+1),2(k+1)}-e_{2k+1,2(k+i)+1}), \; \; \forall i=1,\ldots,n-1,\\
		&A_i=e_{2n+1,2i-1}, \; \; B_i=e_{2n+1,2i}, \; \; \forall i=1,\ldots,n,\\
		&c_h=\sum_{i=0}^{h-2}(-1)^{i+1} e_{2(h-i-1)-1,2(1+i)}, \; \; \forall h=3,5,\ldots,n+1,\\
		&b_h=\sum_{i=0}^{2n-h}(-1)^{i+1} e_{2(n-i),2(h-n+i)-1}, \; \; \forall h=n+1,n+3,\ldots,2n-1,
		\end{align*}
		and
		\begin{align*}
		&[x,B_i]=B_i, \; [y,A_i]=A_i, \; \; \forall i=1,\ldots,n,\\
		&[E_i,B_k]=-B_{k-2i}, \; \; k > 2i,\\
		&[E_i,A_k]=A_{k+2i}, \; \;  k+2i \leq 2n,\\
		&[x,c_{h}]=c_{h}, \; [y,c_{h}]=-c_h, \; \; \forall h=3,5,\ldots,n+1,\\
		&[x,b_{h}]=-b_{h}, \; [y,b_{h}]=b_h, \; \; \forall h=n+1,n+3,\ldots,2n-1,\\
		&[A_i,c_k]=(-1)^{i+1}B_{k-i}, \; [B_i,b_k]=(-1)^{i+1}A_{k-i}, \; \; 1 \leq k-i \leq n,\\
		&[c_k,b_h]=E_{h-k}, \; \; h-k \geq 1,\\
		&[c_{n+1},b_{n+1}]=-x+y,\\
		&[E_{2i},c_h]=-2c_{h-2i}, \; \; h-2i >0,\\
		&[E_{2i},b_h]=2b_{h+2i}, \; \;  h+2i \leq 2n.
		\end{align*}
		The commutator ideal of $\operatorname{Der}(\mathfrak{k}_{n})$ has basis
		$$
		\lbrace x-y,E_2,E_4,\ldots,E_{n-2},c_3,c_5,\ldots,c_{n+1},b_{n+1},b_{n+3}\ldots,b_{2n-1},A_1,\ldots,A_n,B_1,\ldots,B_n \rbrace 
		$$
		and, as in the case of the Heisenberg Leibniz algebra $\mathfrak{l}_{2n+1}^{J_0}$ with $n$ odd, we have that the Lie algebra of derivations is not solvable. The Levi decomposition is
		$$
		\operatorname{Der}(\mathfrak{k}_{n})=R \rtimes S,
		$$
		where
		$$
		R=\langle x+y,E_1,\ldots,E_{n-1},c_3,c_5,\ldots,c_{n-1},b_{n+3},b_{n+5}\ldots,b_{2n-1},A_1,\ldots,A_n,B_1,\ldots,B_n \rangle
		$$
		is the radical and
		$$
		S=\langle x-y,c_{n+1},b_{n+1}\rangle
		$$
		is a Levi complement. Moreover the nilradical of $		\operatorname{Der}(\mathfrak{k}_{n})$ is 
		$$
		N=\langle E_1,\ldots,E_{n-1},c_3,c_5,\ldots,c_{n-1},b_{n+3},b_{n+5}\ldots,b_{2n-1},A_1,\ldots,A_n,B_1,\ldots,B_n \rangle.
		$$
		
		\item If $n$ is odd, then the algebra of derivations of $\mathfrak{k}_{n}$ has dimension $4n$ and it is generated by
		$$
		\lbrace x,y,E_1,\ldots,E_{n-1},c_3,c_5,\ldots,c_{n},b_{n+2},b_{n+4}\ldots,b_{2n-1},A_1,\ldots,A_n,B_1,\ldots,B_n \rbrace.
		$$
		The Lie brackets are the same of the ones listed before when $n$ is even, except for the facts that $x-y$ does not belong to the commutator ideal and $[B_i,b_k]=(-1)^{i}A_{k-i}$, for any $1 \leq k-i \leq n$. In this case the commutator ideal is the subspace generated by
		$$
		\lbrace E_2,E_4,\ldots,E_{n-1},c_3,c_5,\ldots,c_{n},b_{n+2},b_{n+4}\ldots,b_{2n-1},A_1,\ldots,A_n,B_1,\ldots,B_n \rbrace
		$$
		and we have a $(\frac{n+1}{2}+1)-$step solvable Lie algebra with nilradical
		$$
		N=\langle E_1,\ldots,E_{n-1},c_3,c_5,\ldots,c_{n},b_{n+2},b_{n+4}\ldots,b_{2n-1},A_1,\ldots,A_n,B_1,\ldots,B_n \rangle.
		$$
	\end{itemize}
In both cases $n$ is odd or even,  the center $\operatorname{Z}(\operatorname{Der}(\mathfrak{l}_{2n+1}^{J_0}))$ is trivial, the Lie algebra of inner derivations is 
$$
\operatorname{Inn}(\mathfrak{k}_n)=\langle A_1,\ldots,A_n,B_1,\ldots,B_n \rangle \cong \F^{2n},
$$
since
$$
\operatorname{ad}_{e_i}=B_{i-1}+B_i, \; \operatorname{ad}_{f_i}=A_i+A_{i+1}, \; \; \forall i=1,\ldots,n
$$
and
$$
\operatorname{Der}(\mathfrak{h}_{2n+1}) \supseteq 
\operatorname{Der}(\mathfrak{k}_{n}) \supseteq \operatorname{Der}(\mathfrak{l}_{2n+1}^{J_a}),
$$
for any $a \neq 0$. More precisely
$$
\operatorname{Der}(\mathfrak{l}_{2n+1}^{J_0}) \cap \operatorname{Der}(\mathfrak{k}_{n}) = \operatorname{Der}(\mathfrak{l}_{2n+1}^{J_a}).
$$

	\section{Derivations of the Dieudonné Leibniz algebra $\mathfrak{d}_n$}

    Finally we study the derivations of the \emph{Dieudonné Leibniz algebra} $\mathfrak{d}_n$. We fix the basis  $\lbrace e_1,\ldots,e_{2n+1},z\rbrace$ of $\mathfrak{d}_n$.

\begin{thm}
	A derivation of $\mathfrak{d}_n$ is represented by a $(2n+2) \times (2n+2)$ matrix
	\begin{equation}
		$$\[
		\left(
		\begin{array}{c|c|c}
			&  & 0 \\
			\alpha I_{n+1} & C & \vdots \\
			& & 0 \\
			\hline
			&  & 0 \\
			0 & \beta I_n & \vdots \\
			& & 0 \\
			\hline
			\mu_1 \ldots \mu_{n+1} & \nu_1 \ldots  \nu_n & \alpha+\beta
		\end{array}
		\right)
		\]$$
	\end{equation}
	where the $(n+1) \times n$ matrix $C$ is
	$$
	\begin{pmatrix}
		\alpha_1 & 0 & \alpha_2 & 0 & \alpha_3 & 0 & \cdots & \alpha_{\frac{n}{2}} & 0\\
		0 & -\alpha_2 & 0 & -\alpha_3 & 0 & -\alpha_4 & \cdots & 0 & -\alpha_{\frac{n}{2}+1} \\
		\alpha_2 & 0 & \alpha_3 & 0 & \alpha_4 & 0 & \cdots & \alpha_{\frac{n}{2}+1} & 0\\
		0 & -\alpha_3 & 0 & -\alpha_4 & 0 & -\alpha_5 & \cdots & 0 & -\alpha_{\frac{n}{2}+2} \\
		\vdots & \vdots & \vdots & \vdots & \vdots & \vdots &  & \vdots & \vdots \\
		0 & -\alpha_{\frac{n}{2}+1} & 0 & -\alpha_{\frac{n}{2}+2} & 0 & -\alpha_{\frac{n}{2}+3} & \cdots & 0 & -\alpha_n \\
		\alpha_{\frac{n}{2}+1} & 0 & \alpha_{\frac{n}{2}+2} & 0 & \alpha_{\frac{n}{2}+3} & 0 & \cdots & \alpha_n & 0 \\
	\end{pmatrix}
	$$ 
	if $n$ is even and
	$$
	\begin{pmatrix}
		\alpha_1 & 0 & \alpha_2 & 0 & \alpha_3 & 0 & \cdots & 0 & \alpha_{\frac{n+1}{2}}\\
		0 & -\alpha_2 & 0 & -\alpha_3 & 0 & -\alpha_4 & \cdots & -\alpha_{\frac{n+1}{2}} & 0 \\
		\alpha_2 & 0 & \alpha_3 & 0 & \alpha_4 & 0 & \cdots & 0 & \alpha_{\frac{n+1}{2}+1}\\
		0 & -\alpha_3 & 0 & -\alpha_4 & 0 & -\alpha_5 & \cdots & -\alpha_{\frac{n+1}{2}+1} & 0 \\
		\vdots & \vdots & \vdots & \vdots & \vdots & \vdots & & \vdots & \vdots \\
		\alpha_{\frac{n+1}{2}} & 0 & \alpha_{\frac{n+1}{2}+1} & 0 & \alpha_{\frac{n+1}{2}+2} & 0 & \cdots & 0 & \alpha_n \\
		0 & -\alpha_{\frac{n+1}{2}+1} & 0 & -\alpha_{\frac{n+1}{2}+2} & 0 & -\alpha_{\frac{n+1}{2}+3} & \cdots & -\alpha_n & 0 \\
	\end{pmatrix}
	$$ 
	if $n$ is odd.
\end{thm}

The Lie algebra $\operatorname{Der}(\mathfrak{d}_n)$ has dimension $3n+3$ and basis
$$
\lbrace x,y,E_1,\ldots,E_n,A_1,\ldots,A_{2n+1}\rbrace,
$$
where
\begin{align*}
	&x=\sum_{i=1}^{n+1} e_{ii} +e_{2n+1,2n+1}, \;\\ &y=\sum_{i=n+2}^{2n+2} e_{ii}, \; A_i=e_{2n+1,i}, \; \; \forall i=1,\ldots,2n+1,\\
	&E_i=\sum_{k=1}^{2i-1}(-1)^{k+1}e_{k,n+2i+1-k},\; \; \forall i=1,\ldots,\Big\lfloor \frac{n+1}{2} \Big\rfloor,\\
	&E_{\frac{n}{2}+j}=\sum_{k=1}^{n+2-2j}(-1)^{k+1}e_{n+2-k,n+2j-1+k},\; \; \forall j=1,\ldots, \frac{n}{2},
\end{align*}
if $n$ even, and
$$
E_{\frac{n+1}{2}+j}=\sum_{k=1}^{n+1-2j}(-1)^{k}e_{n+2-k,n+2j+k},\; \; \forall j=1,\ldots, \frac{n-1}{2},
$$
if $n$ is odd. The non-zero Lie brackets are
\begin{align*}
	&[x,E_i]=[E_i,y]=\alpha_i, \; \; \forall i=1,\ldots,n,\\
	&[y,A_h]=A_h, \; \; \forall h=1,\ldots,n+1,\\
	&[x,A_k]=A_k, \; \; \forall k=n+2,\ldots,2n+1,\\
	&[A_i,E_k]=\varepsilon_j A_j, \; \; \forall i=1,\ldots,n+1,
\end{align*}
where $\varepsilon_j=\pm 1$ is the only entry different than zero in the $i^{th}$ row of the matrix $\alpha_k$ and $j \in \lbrace n+2,\ldots,2n+1 \rbrace$ is its column. Thus  $\operatorname{Der}(\mathfrak{d}_n)$ is a $3-$step solvable Lie algebra with commutator ideal consisting of the matrices
\begin{equation}
	$$\[
	\left(
	\begin{array}{c|c|c}
		&  & 0 \\
		0 & C & \vdots \\
		& & 0 \\
		\hline
		&  & 0 \\
		0 & 0 & \vdots \\
		& & 0 \\
		\hline
		\mu_1 \ldots \mu_{n+1} & \nu_1 \ldots  \nu_n & 0
	\end{array}
	\right)
	\]$$
\end{equation}
The derived series is
$$
\operatorname{Der}(\mathfrak{d}_n)  \supseteq \langle E_1,\ldots,E_n,A_1,\ldots, A_{2n+1}\rangle \supseteq \langle A_{n+2},\ldots,A_{2n+1} \rangle \supseteq 0
$$
and the nilradical coincides with the commutator ideal (which is a two-step nilpotent Lie algebra). Finally $\operatorname{Z}(\operatorname{Der}(\mathfrak{d}_n))$ is trivial and the left adjoint maps are
$$
\operatorname{ad}_{e_1}=A_{n+2}, \; \operatorname{ad}_{e_i}=A_{n+i}+A_{n+i+1}, \; \; \forall i=2,\ldots,n,
$$
$$
\operatorname{ad}_{e_{n+1}}=A_{2n+1}, \; \operatorname{ad}_{e_j}=A_{j-n}-A_{j-n-1}, \; \; \forall j=n+2,\ldots,2n+1,
$$
thus the inner derivations of the Dieudonné algebra $\mathfrak{d}_n$ are represented by the matrices of the form
\begin{equation}
	$$\[
	\left(
	\begin{array}{c|c}
		& 0 \\
		\textbf{\Large0}   & \vdots \\
		& 0 \\
		\hline
		\mu_1 \; \mu_2 \ldots \mu_n \; \mu \; \nu_1 \ldots \nu_{n} & 0
	\end{array}
	\right)
	\]$$
\end{equation} 
where $\mu=-\displaystyle\sum_{k=1}^{n}\mu_k$. More precisely, the matrix
\begin{equation}
	$$\[
	\left(
	\begin{array}{c|c}
		& 0 \\
		\textbf{\Large0}   & \vdots \\
		& 0 \\
		\hline
		0 \ldots 0 \; \mu_{n+1} \; 0 \ldots 0 & 0
	\end{array}
	\right)
	\]$$
\end{equation} 
does not represent an inner derivation, for every $a_{n+1} \neq 0$.\\
\\ For example, we study $\operatorname{Der}(\mathfrak{d}_n)$ in the case that $n \leq 3$.

\begin{ex}
	If $n=1$, then 
	$$
	D=\begin{pmatrix}[ccc|c]
		\alpha & 0 & \alpha_1 & 0 \\
		0 & \alpha & 0 & 0  \\
		0 & 0 & \beta & 0 \\
		\hline
		\mu_1 & \mu_2 & \nu_1 & \alpha+\beta \\
	\end{pmatrix}
	$$
	thus $\operatorname{Der}(\mathfrak{d}_1)$ is the six-dimensional solvable Lie algebra with basis
	$$
	\lbrace x,y,E,A_1,A_2,A_3 \rbrace,
	$$
	where $x=e_{11}+e_{22}+e_{44},y=e_{33}+e_{44},E=e_{1,3}$ and $A_i=e_{4,i}$, for any $i=1,2,3$, and with non-trivial commutators
	$$
	[x,E]=[E,y]=E, \; [x,A_3]=A_3,\;[y,A_1]=A_1,\;[y,A_2]=A_2,\;[A_1,E]=A_3.
	$$
\end{ex}

\begin{ex}
	If $n=2$, the derivations of $\mathfrak{d}_2$ are of the form
	$$
	\begin{pmatrix}[ccccc|c]
		\alpha & 0 & 0 & \alpha_1 & 0 & 0 \\
		0 & \alpha & 0 & 0  & -\alpha_2 & 0 \\
		0 & 0 & \alpha & \alpha_2 & 0 & 0 \\
		0 & 0 & 0 & \beta & 0 & 0 \\
		0 & 0 & 0 & 0 & \beta & 0 \\
		\hline
		\mu_1 & \mu_2 & \mu_3 & \nu_1 & \nu_2 & \alpha+\beta \\
	\end{pmatrix}
	$$
	and $\operatorname{Der}(\mathfrak{d}_2)$ is a nine-dimensional Lie algebra with commutator ideal consisting of the matrices
	$$
	\begin{pmatrix}[ccccc|c]
		0 & 0 & 0 & \alpha_1 & 0 & 0 \\
		0 & 0 & 0 & 0  & -\alpha_2 & 0 \\
		0 & 0 & 0 & \alpha_2 & 0 & 0 \\
		0 & 0 & 0 & 0 & 0 & 0 \\
		0 & 0 & 0 & 0 & 0 & 0 \\
		\hline
		\mu_1 & \mu_2 & \mu_3 & \nu_1 & \nu_2 & 0 \\
	\end{pmatrix}
	$$
\end{ex}
\begin{ex}
	If $n=3$, then
	\begin{equation}
		$$\[D=
		\left(
		\begin{array}{c|c|c}
			&  & 0 \\
			\alpha I_4 & C & \vdots \\
			& & 0 \\
			\hline
			&  & 0 \\
			0 & \beta I_3 & \vdots \\
			& & 0 \\
			\hline
			\mu_1 \; \mu_2 \; \mu_3 \; \mu_{4} & \nu_1 \; \nu_2\;  \nu_3 & \alpha+\beta
		\end{array}
		\right)
		\]$$
	\end{equation}
	where
	$$
	C=\begin{pmatrix}
		\alpha_1 & 0 & \alpha_2 \\
		0 & -\alpha_2 & 0 \\
		\alpha_2 & 0 & \alpha_3 \\
		0 & -\alpha_3 & 0 \\
	\end{pmatrix}
	$$
	and the Lie algebra $\operatorname{Der}(\mathfrak{d}_3)$ has dimension $9$ with generators
	$$
	\lbrace x,y,E_1,E_2,E_3,A_1,\ldots,A_8 \rbrace,
	$$
	where 
	\begin{align*}
		&x=e_{11}+e_{22}+e_{33}+e_{44}+e_{88}, \; y=e_{55}+e_{66}+e_{77}+e_{88},\\
		&E_1=e_{1,5}, \; E_2=e_{1,7}-e_{2,6}+e_{3,5},E_3=e_{3,7}-e_{2,8},\\
		&A_i=e_{8,i}, \; \; \forall i=1,\ldots,7
	\end{align*}
	and with Lie brackets
	\begin{align*}
		&[x,E_i]=[E_i,y]=E_i, \; \; \forall i=1,2,3,\\
		&[y,A_h]=A_h, \; [x,A_k]=A_k \; \; \forall h=1,2,3,4,\;\forall k=5,6,7,\\
		&[A_1,E_1]=A_5,\;[A_1,E_2]=A_7,\;[A_2,E_2]=-A_6,\\
		&[A_3,E_2]=A_5,\;[A_3,E_3]=A_7,\;[A_4,E_3]=-A_6.
	\end{align*}

\end{ex}

\section{Almost inner derivations of nilpotent Leibniz algebras with one-dimensional commutator ideal}

We recall that a derivation $d$ of a left Leibniz algebra $L$ is an \emph{almost inner derivation} if $d(x) \in [L,x]$, for every $x \in L$. The set of all almost inner derivations of $L$ forms a Lie subalgebra of $\operatorname{Der}(L)$, denoted by $\operatorname{AIDer}(L)$, containing the ideal $\operatorname{Inn}(L)$ of inner derivations of $L$.\\

Derivations of two-step nilpotent Lie algebras were studied in \cite{Burde} and \cite{BURDE2021185} by D. Burde, K. Dekimpe and B. Verbeke. They proved that every almost inner derivations of a Lie algebra of \emph{genus} 1 (i.e.\ with one-dimensional commutator ideal) is an inner derivation. We want to generalize this result in the frame of Leibniz algebras.

\begin{prop}
	Let $L$ be a complex nilpotent Leibniz algebra with $[L,L]=\mathbb{C} z$, such that $L \not\cong \mathfrak{l}_{2n+1}^{J_{\pm 1}}$ and $L \not\cong \mathfrak{d}_{n}$. Then every almost inner derivation $d \in \operatorname{AIDer}(L)$ is an inner derivation.
\end{prop}

\begin{proof}
	Let $d \in \operatorname{AIDer}(L)$. Then $d(y) \in [L,y] \subseteq [L,L]=\mathbb{C} z$, for any $y \in L$ and $d([L,L])=0$. Fixed a basis $\lbrace e_1,\ldots,e_{t-1},z\rbrace$ of $L$, where $t=\dim_{\mathbb{C}}L$, we have that
	$$
	d(e_i)=a_iz, \; \; \forall i=1,\ldots,t-1,
	$$ 
	with $a_i \in \mathbb{F}$ and $d(z)=0$, thus $d \in \operatorname{Inn}(L)$.
\end{proof}

For the Heisenberg Leibniz algebras $\mathfrak{l}_{2n+1}^{J_{a}}$ with $a=\pm 1$ (in \cite{LM2022} it was proved that these two algebras are isomorphic) and for the Dieudonné Leibniz algebra $\mathfrak{d}_n$, it is possible to define an almost inner derivation $d$ which is not inner. For instance, if $a=1$ and we fixed the basis $\lbrace e_1,f_1,\ldots,e_n,f_n,z \rbrace$ of $\mathfrak{l}_{2n+1}^{J_{1}}$, then the matrix 
\begin{equation}
	$$\[
	\left(
	\begin{array}{c|c}
		& 0 \\
		\textbf{\Large0}   & \vdots \\
		& 0 \\
		\hline
		1 \; 0 \ldots 0 \; 0 & 0 \\
	\end{array}
	\right)
	\]$$
\end{equation}
defines a derivation $d \in \operatorname{AIDer}(\mathfrak{l}_{2n+1}^{J_{1}}) \setminus \operatorname{Inn}(\mathfrak{l}_{2n+1}^{J_{1}})$. In the same way 
\begin{equation}
	$$\[
	\left(
	\begin{array}{c|c}
		& 0 \\
		\textbf{\Large0}   & \vdots \\
		& 0 \\
		\hline
		0 \; 0 \ldots 0 \; 1 & 0 \\
	\end{array}
	\right)
	\]$$
\end{equation}
is an almost inner but non-inner derivation of $\mathfrak{l}_{2n+1}^{J_{-1}}$. More precisely every almost inner derivation of $\mathfrak{l}_{2n+1}^{J_{\pm 1}}$ is of the form
\begin{equation}
	$$\[
	\left(
	\begin{array}{c|c}
		& 0 \\
		\textbf{\Large0}   & \vdots \\
		& 0 \\
		\hline
		\mu_1 \; \nu_1 \ldots \mu_n \; \nu_n & 0 \\
	\end{array}
	\right)
	\]$$
\end{equation}
with $a_1,b_1,\ldots,a_n,b_n \in \mathbb{C}$, meanwhile the inner derivations are represented by the set of matrices 
\begin{equation}
	$$\[
	\left(
	\begin{array}{c|c}
		& 0 \\
		\textbf{\Large0}   & \vdots \\
		& 0 \\
		\hline
		0 \; \nu_1 \; \mu_2 \; \nu_2 \ldots \mu_n \; \nu_n & 0 \\
	\end{array}
	\right)
	\]$$
\end{equation}
for $\mathfrak{l}_{2n+1}^{J_{1}}$, and by 
\begin{equation}
	$$\[
	\left(
	\begin{array}{c|c}
		& 0 \\
		\textbf{\Large0}   & \vdots \\
		& 0 \\
		\hline
		\mu_1 \; \nu_1 \; \mu_2 \; \nu_2 \ldots \mu_n \; 0 \; 0 & 0 \\
	\end{array}
	\right)
	\]$$
\end{equation}
for the Leibniz algebra $\mathfrak{l}_{2n+1}^{J_{-1}}$. Finally $\operatorname{AIDer}(\mathfrak{d}_n)$ consists of the matrices of the type
\begin{equation}
	$$\[
	\left(
	\begin{array}{c|c}
		& 0 \\
		\textbf{\Large0}   & \vdots \\
		& 0 \\
		\hline
		\mu_1 \ldots \mu_{n+1}\; \nu_1 \ldots \nu_n & 0
	\end{array}
	\right)
	\]$$
\end{equation} 
and an example of almost inner but non-inner derivations is given by the linear map $d \in \operatorname{gl}(\mathfrak{d}_n)$ defined by $d(e_{n+1})=z$.

%\section*{Declarations}
%
%Not applicable. There is no Competing Interest.

\printbibliography

\end{document}